\documentclass[
]{mpi2015-cscpreprint}

\usepackage[american]{babel}

\usepackage{graphicx}
\usepackage{subcaption}

\usepackage{amssymb}
\usepackage{amsthm}
\usepackage{pgfplots}
\pgfplotsset{compat=newest}
\usepgfplotslibrary{groupplots}

\definecolor{mycolor}{rgb}{0,0.6,0.5}
\allowdisplaybreaks

\newtheorem{theorem}{Theorem}[section]

\newtheorem{lemma}{Lemma}[section]
\newtheorem{example}{Example}[section]

\newtheorem{remark}{Remark}[section]
\newtheorem{definition}{Definition}[section]

\theoremstyle{definition}\newtheorem{assumption}{Assumption}
\Crefname{assumption}{Assumption}{Assumptions}

\DeclareMathOperator{\diag}{diag}
\DeclareMathOperator{\Id}{Id}
\DeclareMathOperator{\Zero}{O}

\def \G {\Gamma}
\def \U {\Upsilon}
\def \td {\tilde}

\def \x {y} 
\def \y {z} 

\def \P {R} 
\def \Q {S} 
\def \R {T} 

\def \p {r} 
\def \q {s} 
\def \r {t} 

\def \a {q} 
\def \A {Q} 
\def \f {\mathcal{F}}
\def \e {\zeta} 
\def \u {u}
\def \t {\eta} 

\begin{document}


\title{Exponential Lag Synchronization of Cohen-Grossberg Neural Networks
with Discrete and Distributed Delays on Time Scales}
  
\author[]{Vipin Kumar}
\affil[]{Max Planck Institute for Dynamics of Complex Technical Systems, Sandtorstra\ss e 1, 39106 Magdeburg, Germany.\authorcr
  \email{vkumar@mpi-magdeburg.mpg.de}, \orcid{0000-0002-7068-5426}}
  
\author[]{Jan Heiland}
\affil[]{Max Planck Institute for Dynamics of Complex Technical Systems, Sandtorstra\ss e 1, 39106 Magdeburg, Germany.\authorcr
  \email{heiland@mpi-magdeburg.mpg.de\}}, \orcid{0000-0003-0228-8522}}
  
\author[]{Peter Benner}
\affil[]{Max Planck Institute for Dynamics of Complex Technical Systems, Sandtorstra\ss e 1, 39106 Magdeburg, Germany.\authorcr
  \email{benner@mpi-magdeburg.mpg.de}, \orcid{0000-0003-3362-4103}}

\shorttitle{Cohen-Grossberg Neural Networks on Time Scales}
\shortauthor{V. Kumar, J. Heiland, P. Benner}
\shortdate{}
  

  
\abstract{
In this article, we investigate exponential lag synchronization results for the Cohen-Grossberg neural networks (C-GNNs) with discrete and distributed delays on an arbitrary time domain by applying feedback control. 
We formulate the problem by using the time scales theory so that the results can be applied to any uniform or non-uniform time domains. 
Also, we provide a comparison of results that shows that obtained results are unified and generalize the existing results. 
Mainly, we use the unified matrix-measure theory and Halanay inequality to establish these results. 
In the last section, we provide two simulated examples for different time domains to show the effectiveness and generality of the obtained analytical results.}

\novelty{This is the first attempt to discuss the exponential lag synchronization results for the generalized C-GNNs with mixed delays on time scales. 
The results are obtained by applying the novel unified matrix-measure theory and Halanay inequality. 
A comparison of results shows that these results unify and generalize the existing results. 
An example with simulation for different time domains is given to illustrate the analytical results.}

\maketitle

\section{Introduction}\label{sec1}
Since the 1980s, neural networks (NNs), including 
recurrent NNs, Hopfield NNs, cellular NNs, and bi-directional associative NNs, 
have been a subject of intense study because of their large number of potential applications in many fields, such as 
the classification of patterns, 
signal and image processing, optimization problems, associative memory, parallel computing, and so on.
In $1983$, Cohen-Grossberg \cite{CG} introduced the C-GNNs which are recognized as one of the most important and typical NNs
because some other well-known NNs, for example, recurrent NNs, cellular NNs, and Hopfield NNs are special cases of C-GNNs.  
As a result, these types of networks have attracted considerable research attention and have been extensively studied 
in terms of their dynamical 
properties such as state estimation \cite{CG-app-state}, periodicity \cite{CG-app-periodic}, stability \cite{CG-app-stability-2,CG-app-stability-3},  boundedness \cite{CG-app-bound-2}, and synchronization \cite{CG-synchro-1,syn-exponential-1}.
Furthermore, due to the importance of discrete-time C-GNNs as discussed in \cite{CG-D-1}, the dynamics of discrete-time C-GNNs have become a popular research topic; see, for example, \cite{CG-D-2, CG-D-3,CG-D-6,CG-D-7}.

Synchronization is one of the most important qualitative properties of dynamic systems and means that two or more dynamic systems lead to a common dynamical behaviour by using some coupling or external forces.  
The concept of synchronization in drive-response systems was first introduced by Pecora and Carrol \cite{pecora}, and since then, it has been capturing increased attention   from   both   a fundamental   and   application-driven   perspective.
  Potential applications   of synchronization   can be found in many areas of applied sciences, 
such as harmonic oscillation generation, information science, human heartbeat regulation, 
chemical and biological systems, and secure communication \cite{app-synchro-1,app-synchro-2,app-synchro-3}.
In the last few years, various types of synchronization phenomena have been discovered and investigated, such as exponential synchronization \cite{syn-exponential-2, add-2}, complete synchronization \cite{syn-complete},  finite-time synchronization \cite{syn-finite,syn-finite-2}, lag synchronization \cite{syn-lag}, adaptive synchronization \cite{syn-adaptive,add-1}, and projective synchronization \cite{syn-projective,add-2}.  
Among them, lag synchronization has been extensively studied \cite{work-lag-synchro-1,work-lag-synchro-2,work-lag-synchro-3,work-lag-synchro-4}
due to its relevance in connected electronic networks, where constant time shifts between drive and response systems can make complete synchronization difficult to implement effectively.
 In practical applications, both discrete and continuous dynamic systems play a significant role, but results for them are often studied separately.
 
In $1988$, Hilger \cite{ts-thesis}, introduced the so-called \emph{time scale theory (or measure chain theory)} 
which unifies the   separate analysis of   discrete and continuous   dynamic systems into a single comprehensive analysis.
  Eventhough, the study of dynamic systems is not limited to just discrete and continuous-time domains. In fact,
there are many other time domains which can be useful to study the dynamic behaviours of dynamic systems more accurately. 
For example, 
to model the growth process of some species like Magicicada Septendecim, Magicicada Cassini,
and Pharaoh Cicada, we need a time domain of the form 
$T=\cup_{k=0}^\infty [k(a+b), k(a+b)+b], \ a,b \in (0,\infty)$.
 Further, 
there exist neurons in the brain that follow a pattern of being active during the day and inactive at night. Intuitively, the dynamic behaviour of these neurons can be observed in the time domain 
$\mathbb{T} = \bigcup_{l=0}^\infty [24l, 24l+d_l],$
where $d_l$ denotes the number of active hours of the neurons in each day; 
see \Cref{ex-brain}.
\begin{figure}[t]
\begin{center}
\includegraphics{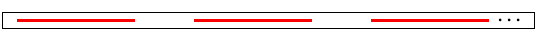}
 \caption{Red lines denote the active time of neurons during a day while the gap shows
the inactive time of neurons at night} \label{ex-brain}
\end{center}
\end{figure}
Another example is an RLC circuit (see \Cref{RLC}), where if the capacitor discharges with small time units $\delta>0$ at periodic intervals of $l$ time units, the dynamics of such a model can be modelled on the time scale $T = \bigcup_{l=0}^\infty [l, l+1-\delta].$ 
\begin{figure}[t]
  \begin{center}
    \includegraphics[width=3.in, height=1.4in]{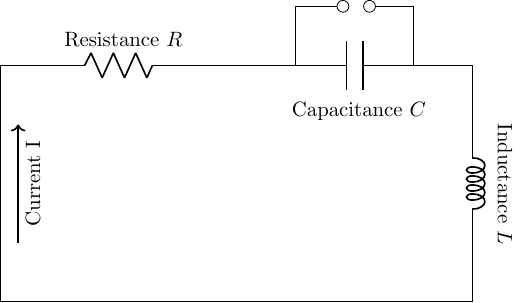}
  \caption{A simple $RLC$ circuit} \label{RLC}
  \end{center}
\end{figure}
These examples require a time domain 
  which is neither discrete nor continuous.
  However,   
the time scale theory can overcome such difficulties as it
gives the freedom to work on the general domain, 
i.e., the results obtained by using the time scales will also be valid for uniform and non-uniform time domains such as the non-overlapping closed intervals, a mixture of closed intervals and discrete points, and even a discrete non-uniform time domain.
Thus, we can summarize the above and state that ``Unification and Extension" are two main features of the time scale theory. 
Therefore, it is worth to investigate the dynamic equations on time scales. 
For more 
  studies   on time scales, one can refer to the monograph \cite{ts-book1}.

In the last few years, the study of dynamic equations on time scales has drawn a tremendous amount of
attention across the world and many researchers found its applications in many fields, such as
epidemiology, economics, and control theory \cite{ts-app-eco, ts-app-control}.
Recently, many authors   have also   established different types of
qualitative behaviours of dynamic systems on time scales, for example, the existence of solutions, stability analysis, stabilization, and synchronization \cite{work-ts-1,work-ts-2,work-ts-3,work-ts-5,work-ts-6}.
Also, few authors established the existence of periodic, anti-periodic, almost-periodic solutions and their stability results of the C-GNNs \cite{work-cgnn-1,work-cgnn-2,work-cgnn-3,E-work-cgnn-1,E-work-cgnn-2,E-work-cgnn-3}.
In \cite{E-work-cgnn-1}, the authors studied the existence of   an   anti-periodic solution and exponential stability for C-GNNs with time-varying delays on time scales. 
In \cite{E-work-cgnn-2}, the authors established the existence and global exponential stability of almost periodic solutions for C–GNNs with distributed delays on time scales while 
in \cite{E-work-cgnn-3}, the authors considered the impulsive C-GNNs
with distributed delays on time scales and studied the existence and exponential stability of periodic solutions by 
using Lyapunov functions, M-matrix theory, and coincidence degree theory.

Despite the growing interest in the study of dynamic equations on time scales,
  the synchronization problem of C-GNNs on time scales has not been studied so far to the best of our knowledge. 
Therefore, to fill this gap, in this work, we  establish exponential lag synchronization results for 
C-GNNs with discrete and distributed time delays on time scales by using feedback control, 
a novel unified matrix-measure 
technique
and  
the Halanay inequality.
In short, the 
  main focus and benefit of this manuscript can be summarized as follows:
 
\begin{itemize}
\item The C-GNNs with discrete and distributed delays on arbitrary time domains are considered to study exponential lag synchronization. 

\item The problem is formulated by using the time scales theory and
the results are derived based on a unified matrix-measure theory and the Halanay inequality.

\item The results for different special cases are given which shows that the obtained results unify and generalize the existing results.

\item A simulated example for different time scales including continuous, discrete and non-overlapping closed intervals, is given to verify the obtained analytical outcomes.
\end{itemize}

The remaining part of the manuscript is organized as follows: 
In  \Cref{sec:pre}, we recall 
  basic concepts from matrix theory and time scales that are essential for the subsequent sections.
In  \Cref{sec:SOP}, we formulate our statement of the problem.
In  \Cref{sec:Results}, the main results are discussed.
  Finally, in 
\Cref{sec:examples}, two numerical examples with simulation are given to verify the obtained results.

\section{Preliminaries} \label{sec:pre}
Throughout this paper, the notations $\mathbb{R}, \mathbb{Z}$ and $\mathbb{N}$
denote the set of all real, integers. and natural numbers, respectively;
$\mathbb{T}$ denotes the time scale;
$\emptyset$ denotes the empty set;
$\mathbb{R}^n$ and $\mathbb{R}^{n \times m}$
 denote the $n$-dimensional Euclidean space and the set of all $n \times m$ matrices, respectively;
$\diag\{\ldots\}$ denotes the diagonal matrix;
Superscript $*$ denotes the matrix transpose;
$\Id$ and $\Zero$ denote the identity and zero matrices of appropriate dimensions, respectively;
$[a,b]_\mathbb{T} = [a,b] \cap \mathbb{T}$, denotes the time scale interval.
For any $a, b \in \mathbb{R}, C([a,b],\mathbb{R}^n)$ denotes
the set of continuous functions from $[a,b]$ into $\mathbb{R}^n$;
$\| \cdot \|_p, \ (p =1,2,\infty)$ is used to denote the $p$-norm for a vector or for a matrix.

Next, we recall some basic definitions and results about time scale calculus. 

A \emph{time scale} is an arbitrary non-empty closed subset of the real numbers $\mathbb{R}$ with the topology and ordering inherited from $\mathbb{R}$. 
$h\mathbb{Z}(h>0)$, $\mathbb{R},$  $\mathbb{P}_{a,b} = \cup_{k=0}^\infty [k(a+b), k(a+b) + a]$ for $a,b \in (0,\infty)$, and any discrete set are some examples of time scales. 
The \emph{forward and backward jump operators} $\sigma, \rho : \mathbb{T} \to \mathbb{T}$ are defined by
$\sigma(t) = \inf \{ s \in \mathbb{T}: s>t\}$
and 
$\rho(t) = \sup\{ s \in \mathbb{T}: s<t \},$
respectively with the substitution 
$\sup \mathbb{T} = \inf \emptyset$ and $\inf \mathbb{T} = \sup \emptyset$.
Also the graininess functions $\mu:\mathbb{T} \to [0,\infty)$ is given by $\mu(t) = \sigma(t) - t$.
A point $t\in \mathbb{T}$ is called 
\emph{right-dense} if $t<\max\{ \mathbb{T} \}$ and $\sigma(t) = t$,
\emph{left-dense} if $t>\min\{ \mathbb{T} \}$ and $\rho(t) = t$,
\emph{right-scattered} if $\sigma(t)>t$, and
\emph{left-scattered} if $ \rho(t) <t$.
If $\mathbb{T}$ has a left-scattered maximum $M$, then we set  $\mathbb{T}^k = \mathbb{T}\setminus\{M\}$, otherwise $\mathbb{T}^k = \mathbb{T}$.

\begin{definition} [\cite{work-ts-1}, Def. 1]  \label{def:delta} 
Let $f : \mathbb{T} \to \mathbb{R}$ be a function. Then the \emph{delta derivative} of $f$ at a point $t \in \mathbb{T}^k$ is defined as a number $f^\Delta(t)$ (provided it exists) whenever for each $\epsilon > 0$ there exists a neighborhood $U$ of $t$ such that
$$ \vert [f(\sigma(t)) - f(s)] - f^\Delta(t)[\sigma(t) - s] \vert \leq \epsilon \vert \sigma(t) - s  \vert \  \text{for all} \  s \in U.
$$
Further, if the neighborhood $U$ is replaced by the right-hand sided neighborhood $U^+$, then the delta derivative is
called the \emph{upper right Dini-delta-derivative} and denoted by $D_\Delta^+f(t)$.
\end{definition}

\begin{remark} \label{add:remark-1}
In the above \Cref{def:delta}, if $\mu(t) = 0$, then the delta derivative $f^\Delta(t)$ becomes the ordinary derivative $f^\prime(t)$ and the upper right Dini-delta-derivative $D_\Delta^+f(t)$ becomes the ordinary upper right Dini-derivative $D^+f(t)$.
Further, if $\mathbb{T} = h \mathbb{Z}, \ h>0$, then the delta derivative $f^\Delta(t)$ becomes the h-difference operator, i.e., $f^\Delta(t) = \frac{f(t+h)-f(t)}{h}$. 
\end{remark}

\begin{remark}
Let $f : \mathbb{T} \to \mathbb{R}$ is differentiable at $t\in \mathbb{T}^k$, then the forward operator $\sigma$ and the delta derivative of $f$ are related by the formula $f(\sigma(t)) = f(t) + \mu(t) f^\Delta(t)$.
\end{remark}

A function $f :\mathbb{T}\to \mathbb{R}$ is called \emph{regressive (or positive regressive)} if $1+\mu(t)f(t) \neq 0 (or \ >0)$ for all $t\in \mathbb{T}$.
Also, $f$ is called \emph{regulated} provided its right-side limit exists (finite) at all right-dense points of $\mathbb{T}$ and its left-side limit exist (finite) at all left-dense points of $\mathbb{T}$. 
Furthermore, $f$ is called a \emph{rd-continuous function} if it is regulated and it is continuous at all right-dense points of $\mathbb{T}$.
The collection of all rd-continuous functions and rd-continuous regressive (or rd-continuous positive regressive) functions from $\mathbb{T}$ to $\mathbb{R}$ are defined, respectively, by $C_{rd}(\mathbb{T},\mathbb{R})$ and $\mathcal{R}(or \ \mathcal{R}^+)$.

\begin{definition} [\cite{work-ts-3}, Def. 2.6]
For any $p \in \mathcal{R}$ and $t \in \mathbb{T}^k$, we define $\ominus p$ by 
$$(\ominus p)(t) = - \dfrac{p(t)}{1+\mu(t) p(t)}.$$
\end{definition}

\begin{remark}
If $p \in \mathcal{R}$, then $\ominus p \in \mathcal{R}$.
\end{remark}
  
Next, we define the time scales version of the exponential function.   
\begin{definition} [\cite{ts-book1}, Def. 2.30] Let $p\in \mathcal{R}$, then we define the exponential function on time scales by 
$$e_p(t,s)= \exp \left(\int_s^t \zeta_{\mu (z)}(p(z))\Delta   z  \right) \ \text{for} \  t, s   \in \mathbb{T} $$
with 
\[
    \zeta_{\mu (s)}(p(s))=
\begin{cases}
    \dfrac{1}{\mu ( s  )}\log(1+p(s)\mu (s)), & \text{if } \mu (  s  ) \neq 0,\\
   p(s), & \text{if } \mu (s)=0.
\end{cases}
\]
\end{definition}

Next, we define the delta-integral on time scales. 

\begin{definition}  [\cite{ts-book1}, Def. 1.71]
Let $f: \mathbb{T} \to \mathbb{R}$ be a regulated function, then a function $F : \mathbb{T} \to \mathbb{R}$ is called an \emph{anti-derivative of $f$} if $F^\Delta(t) = f(t)$ holds for all $t \in \mathbb{T}^k$. Also, we define the Cauchy integral by
$$\int_{a}^b f(t) \Delta (t) = F(b)-F(a) \ \text{for all \ } a,b, \in \mathbb{T}. $$
\end{definition}

\begin{remark} \label{remark:integral}
For any $a, b \in \mathbb{T}$ and $ f \in  C_{rd}(\mathbb{T},\mathbb{R})$, if we set $\mathbb{T=R}$, then we have
\begin{align*}
\int_a^b f(t) \Delta t = \int_a^b f(t) dt.
\end{align*}
Further, if $[a, b)_\mathbb{T}$ consists of only isolated points, then we have 
\begin{align*}
\int_a^b f(t) \Delta t = \begin{cases}
\sum_{t \in [a,b)_\mathbb{T}} \mu(t) f(t) \quad & \text{if}  \ a<b,\\
0 & \text{if}  \ a=b,\\
 - \sum_{t \in [a,b)_\mathbb{T}} \mu(t) f(t) \quad & \text{if}  \ a>b.
\end{cases}
\end{align*}
\end{remark}

Next, we recall some basics from matrix-measure theory.

\begin{definition} [\cite{work-ts-6}, Def. 1] \label{mm-real}
The \emph{generalized matrix-measure}
 and 
\emph{classical matrix-measure} of a real square matrix 
$W = (w_{kl})_{n \times n}$ with respect to the $p-$norm $(p=1, 2$ or $\infty)$ are defined by
\begin{align*}
\omega_p (W,h) =  \dfrac{\| \Id + h W\|_p - 1}{h} \ \text{and} \ \Lambda_p (W) = \lim_{s \to 0^+} \dfrac{\| \Id + s W\|_p - 1}{s},
\end{align*}
respectively, where $h>0$. 
The matrix norms and corresponding classical matrix-measures are given in \Cref{table:Matrix-norm-measure}.
\end{definition}

\begin{table}[h]  
  \begin{center}
    \begin{tabular}{cl}    
      Matrix norm & \quad Matrix-measure \\
      \hline\noalign{\medskip}
$\|W\|_1 = \max_{j} \sum_{i=1}^n \vert w_{ij} \vert $ & \quad $\Lambda_1(W) =   \max_{j} w_{jj} +  \sum_{i=1, i \neq j}^n \vert w_{ij}\vert$ \\
     $\|W\|_2 = \sqrt{\lambda_{\max}( W^T W)}$ & \quad $\Lambda_2(W) = \dfrac{1}{2}{\lambda_{\max}( W^T + W)}$\\
    $\|W\|_\infty = \max_{i} \sum_{j=1}^n \vert w_{ij} \vert$ & \quad    $\Lambda_\infty(W) = \max_{i} w_{ii} +  \sum_{j=1,  \neq i}^n \vert w_{ij} \vert$   \\
      \noalign{\medskip}\hline\noalign{\smallskip}
    \end{tabular} 
     \caption{Matrix norms and corresponding classical matrix-measures}
     \label{table:Matrix-norm-measure}
  \end{center}
\end{table}

\begin{definition} [\cite{work-ts-6}, Def. 2 ] \label{mm-ts}
  Let $W \in \mathbb R^{n\times n}$ be a real matrix and let $\mathbb{T}$ be an arbitrary
time scale.
Then the \emph{unified matrix-measure on $\mathbb{T}$ with respect to the $p-$norm} $(p=1, 2$ or $\infty)$ is defined as
\begin{align*}
M_p(W,\mathbb{T}) =
 \begin{cases}
  \max\biggl\{ \dfrac{\|\Id + \mu(t) W\|_p - 1}{\mu(t)}\colon t \in
  \mathbb{T}\biggr\},  \text{if } \mu(t) > 0, \forall \ t\in \mathbb{T},\\ 
  \max\biggl\{
    \Lambda_p(W), \ \max\Bigl\{ \dfrac{\|\Id + \mu(t) W\|_p - 1}{\mu(t)}\colon t \in
  \mathbb{T},\mu(t)>0\Bigr\} \biggr\}, \ \text{else.} \end{cases}
\end{align*}
\end{definition}

Note that for $\mathbb{T}= \mathbb{R}$ and $\mathbb{T} = h\mathbb{Z}$, $h>0$, \Cref{mm-ts} reduces to
\Cref{mm-real}.

\section{Statement of Problem} \label{sec:SOP}
We consider a class of C-GNNs with discrete and distributed delays on time scales of the following form:
\begin{align} \label{eq:main-D}
\begin{cases}
\x^\Delta(t)  & =  - \G(\x(t))[\U(\x(t)) - \P \f(\x(t)) - \Q \f(\x(t-\t_1)) - \R \int_{t-\t_2}^t \f(\x(s)) \Delta s  - I], \ t \in [0,\infty)_\mathbb{T},\\
\x(s)  & = \phi(s), \ s \in [-\t , 0]_\mathbb{T}, 
\end{cases}
\end{align}
\sloppy where  $\x(t) = [\x_1(t),\x_2(t),\ldots,\x_n(t)]^* \in \mathbb{R}^n$ is the state vector; 
$ \P = (\p_{ij})_{n \times n} \in \mathbb{R}^{n \times n}, \Q= (\q_{ij})_{n \times n} \in \mathbb{R}^{n \times n}$   and   $ \R= (\r_{ij})_{n \times n} \in \mathbb{R}^{n \times n}$  are the connection, discrete delay connection and distributed delay 
  connection    strength matrices, respectively; 
$\t_1(>0)$ and $\t_2(>0) $ are the discrete and distributed delay, respectively, such that $t-\t_1 \in \mathbb{T}$ and $t-\t_2 \in \mathbb{T}$; $\t = \max\{\t_1,\t_2\}$; 
$\G(\x(t)) = \diag\{\G_1(\x(t)), \G_2(\x(t)), \ldots, \G_n(\x(t))\} \in \mathbb{R}^{n \times n}$ is the state-dependent amplification function;
$\U(\x(t)) = [\U_1(\x(t)),\U_2(\x(t)),\ldots,\U_n(\x(t))]^* \in \mathbb{R}^n$ is the appropriate behaviour function;  
$\f(\x(\cdot)) = [\f_1(\x(\cdot)),\f_2(\x(\cdot)),\ldots,\f_n(\x(\cdot))]^* \in \mathbb{R}^n$  denotes the activation function;
$I$ is the external bias term;
$\phi \in C_{rd}([-\t,0]_\mathbb{T}, \mathbb{R}^n)$.

In this paper, we shall establish synchronization results by using the drive-response technique. Therefore, we consider system \eqref{eq:main-D} as the drive system and, correspondingly, we consider a 
   response system   described    as follows:
\begin{align} \label{eq:main-R}
\begin{cases}
\y^\Delta(t)  &= - \G(\y(t))[\U(\y(t)) - \P \f(\y(t)) - \Q \f(\y(t-\t_1)) \\
& \quad - \R \int_{t-\t_2}^t \f(\y(s)) \Delta s - I] + \u(t), \ t \in [0,\infty)_\mathbb{T},\\
\y(s)  & = \psi(s), \ s \in [-\t , 0]_\mathbb{T}, 
\end{cases}
\end{align}
where $\y(t) \in \mathbb{R}^n$; 
$\psi \in C_{rd}([\t,0]_\mathbb{T}, \mathbb{R}^n)$; 
$\u(t)$ is the control function defined as 
\begin{align} \label{eq:control}
\u(t) = -K (\y(t) - \x(t-\beta)),
\end{align}
where $K$ is 
the feedback gain matrix and $\beta$ is the transmittal delay such that $t-\beta \in \mathbb{T}$. 
\begin{remark} \label{remark:ts:cases}
The considered class of C-GNNs is defined on the general time domain, and hence, it contains the usual continuous-time C-GNNs, discrete-time C-GNNs, and many more. For example, if we consider the\textbf{ continuous-time domain}, i.e., $\mathbb{T} = \mathbb{R}$, then, 
see \Cref{add:remark-1},
the drive system \eqref{eq:main-D} becomes 
\begin{align} \label{eq:main-D-r}
\x^\prime (t)  =  - \G(\x(t))[\U(\x(t)) - \P \f(\x(t)) - \Q \f(\x(t-\t_1)) - \R \int_{t-\t_2}^t \f(\x(s)) d s - I]
\end{align}
and the response system \eqref{eq:main-R} becomes  
\begin{align} \label{eq:main-R-r}
\y^\prime(t)  = \G(\y(t))[\U(\y(t)) - \P \f(\y(t)) - \Q \f(\y(t-\t_1)) - \R \int_{t-\t_2}^t \f(\y(s)) d s - I] + \u(t),  
\end{align}
where  $t \in [0,\infty)$, and the rest of the parameters are the same as defined previously. 
Also, if we choose, the \textbf{$h-$difference discrete-time domain}, i.e., $\mathbb{T}=h\mathbb{Z}$, $h>0$, then, 
see \Cref{add:remark-1} and \Cref{remark:integral},
the drive system \eqref{eq:main-D} is converted to 
\begin{align} \label{eq:main-D-d}
\x(t+h)  & = \x(t) - h \G(\x(t))\bigg [\U(\x(t)) - \P \f(\x(t)) - \Q \f(\x(t-\t_1))  - \R \sum_{k = \frac{t-\t_2}{h}}^{\frac{t}{h} -1}  h\f(\x(kh))  - I  \bigg ] 
\end{align}
and the response system \eqref{eq:main-R} is converted to 
\begin{align} \label{eq:main-R-d}
 \y(t+h)  & =\y(t) - h \G(\y(t)) \bigg [\U(\y(t)) - \P \f(\y(t)) - \Q \f(\y(t-\t_1))  -  \R \sum_{k = \frac{t-\t_2}{h}}^{\frac{t}{h} -1}  h\f(\y(kh))  - I \bigg ] +  h \u(t),
\end{align}
where
 $t \in [0,\infty)_{h \mathbb{Z}}$.
Furthermore, 
by applying the above mentioned cases to 
   
the \textbf{non-overlapping time domain} $\mathbb{T} =\cup_{i=0}^\infty [i, i+h], 0 <h<1,$ the concrete
expression of   the   drive system \eqref{eq:main-D} 
  can be derived as  
\begin{align}
\begin{cases}
\x^\prime (t)   & =  - \G(\x(t))[\U(\x(t)) - \P \f(\x(t)) - \Q \f(\x(t-\t_1))   \\
& \quad   - \R \int_{t-\t_2}^t \f(\x(s)) d s - I], \ t \in \cup_{i=0}^\infty [i, i+h),  \\
\x(t+1-h)   & = \x(t) - (1-h)\G(\x(t))  [\U(\x(t)) - \P \f(\x(t)) - \Q \f(\x(t-\t_1))  \\
   & \quad - \R \sum_{k = \frac{t-\t_2}{1-h}}^{\frac{t}{1-h} -1}  (1-h)\f(\x(k(1-h))) - I  ], \ t = \cup_{i=0}^\infty \{i+h\}
\end{cases}
\end{align}
and the response system \eqref{eq:main-R} 
  can be derived as  
\begin{align}
\begin{cases}
\y^\prime(t)   & = - \G(\y(t))[\U(\y(t)) - \P \f(\y(t)) - \Q \f(\y(t-\t_1))   \\
& \quad - \R \int_{t-\t_2}^t \f(\y(s)) d s - I]  + \u(t), \ t \in \cup_{i=0}^\infty [i, i+h), \\
\y(t+1-h)  & =   \y(t) - (1-h)\G(\y(t))  [\U(\y(t)) - \P \f(\y(t)) - \Q \f(\y(t-\t_1)) \\
 & \quad - \R \sum_{k = \frac{t-\t_2}{1-h}}^{\frac{t}{1-h} -1}  (1-h)\f(\y(k(1-h))) - I  ] + (1-h)\u(t), \\
  & \qquad \  \ t = \cup_{i=0}^\infty \{i+h\}. 
\end{cases}
\end{align}
\end{remark}

The main idea of synchronization is 
that the response system \eqref{eq:main-R} utilizes a feasible controller to synchronize itself with the drive system \cref{eq:main-D}. 
 Mathematically, we can define it in the following definition. 
\begin{definition} \label{main:Def:ELS}
The drive system \eqref{eq:main-D} and the response system \eqref{eq:main-R} are said to be \emph{exponentially lag-synchronized} in the timescale sense
under the control protocol \eqref{eq:control} if there exist two constants $C >0$ and $\nu > 0$ such that the following inequality holds
$$\| \y(t) - \x(t - \beta) \|_p \leq C e_{\ominus \nu}(t,0),  \  t \geq 0.$$
\end{definition}

\begin{remark}
In the above \Cref{main:Def:ELS}, if $\beta = 0$, then 
the drive system \eqref{eq:main-D} and the response system \eqref{eq:main-R} are called exponentially synchronized.
\end{remark}

Now, to prove the synchronization results, 
we define the error between   the   drive system \eqref{eq:main-D}
and   the   response system \eqref{eq:main-R} by
 $\e(t) = \y(t) - \x(t-\beta)$, then the error dynamics can be written as
\begin{align} \label{eq:main-E}
\e^\Delta(t)  & = -K\e(t)  
- \td\G(\e(t))[\td\U(\e(t)) - \P \td \f(\e(t)) - \Q \td \f(\e(t-\t_1))- \R \int_{t-\t_2}^t \td \f(\e(s)) \Delta s - I],
\end{align}
where $\e(t)\in \mathbb{R}^n$ and 
\begin{align*}
\td \G(\e(t)) \td \U(\e(t)) & = \G(\y(t))\U(\y(t)) -\G(\x(t-\beta))\U(\x(t-\beta)), \\
\td \G(\e(t)) \P \td \f(\e(t)) & = \G(\y(t)) \P \f(\y(t)) -\G(\x(t-\beta)) \P \f(\x(t-\beta)), \\
\td \G(\e(t)) \Q \td \f(\e(t-\t_1)) & = \G(\y(t)) \Q \f(\y(t-\t_1)) - \G(\x(t-\beta)) \Q \f(\x(t-\beta-\t_1)), \\
\td \G(\e(t)) \R \int_{t-\t_2}^t \td \f(\e(s)) \Delta s  & = \G(\y(t)) \R \int_{t-\t_2}^t \f(\y(s)) \Delta s - \G(\x(t-\beta)) \R \int_{t-\t_2}^t \f(\x(s-\beta)) \Delta s, \\
\td \G(\e(t)) I & =  \G(\y(t)) I -  \G(\x(t-\beta))I.
\end{align*} 

From the definition of $\e(t)$, 
it is clear that if the error system \eqref{eq:main-E} is exponentially stable, then the drive system \eqref{eq:main-D} and the response system \eqref{eq:main-R} are exponentially lag-synchronized. 
Therefore, our goal is to show the exponential stability of the error system \eqref{eq:main-E}.   

To deal with the lag delay, we set 
$\x(s) = \phi(-\t)$ for all $s \in [-\t-\beta, -\t]_\mathbb{T}$ and  
\begin{align*}
\Psi(s) =  
\begin{cases}
\phi(s), \ s \in [-\t, 0]_\mathbb{T},\\
\phi(-\t), \ s \in [-\t-\beta, -\t]_\mathbb{T},
\end{cases}
\end{align*}
then, we can define the initial condition for the error system \eqref{eq:main-E} as follows
\begin{align*}
\e(s) = \psi(s) - \Psi(s-\beta), \ s \in [-\t,0]_\mathbb{T}.
\end{align*}
In order to prove the main results, we need the following assumption. 
\begin{assumption} [\cite{add-1},   Ass. A1,A2]  
  \label{ass:fjgi-bounded}
 
The functions 
$\G, \U$ and $\f$ are Lipschitz continuous and bounded. In particular,
for any $\x,\y \in \mathbb{R}^n$, there exist positive constants $L_\G, L_\U, L_\f$ such that
\begin{align*}
& \|\G(\x)-\G(\y)\|_p \leq L_{\G} \|\x-\y\|_p, \ \|\U(\x)-\U(\y)\|_p \leq L_{\U} \| \x-\y\|_p, \  \|\f(\x)-\f(\y)\|_p \leq L_{\f} \|\x-\y\|_p. 
\end{align*}
Also, there exist positive constants $M_\G, M_\U, M_\f$ such that
\begin{align*}
\|\G(\x)\|_p \leq M_{\G}, \ \|\U(\x)\|_p \leq M_{\U} , \ \|\f(\x)\|_p \leq M_{\f}. 
\end{align*}
 
\end{assumption}
 
We note that typical choices of the activation functions like $\tanh$ or sigmoid fulfil this assumption.
Moreover, anticipating the nature of the estimates of the following section, we can state the following relaxation
of \Cref{ass:fjgi-bounded}.
\begin{remark} \label{add:remark:2}
If the states can be confined a priori to a bounded set $\Omega \subset \mathbb{R}^n$, then the Lipschitz and boundedness conditions need to be established on $\Omega$ only.
\end{remark}
 

\section{Exponential Lag Synchronization Results} \label{sec:Results}
In this section, we provide the main results of this manuscript. 
Before that, we are giving an important lemma which is useful to establish these results.

\begin{lemma} [\cite{work-ts-6}, Lemma 2]
\label{lemma-1}
For any real scalars $c$ and $d$ such that $c>d>0$ and $-c \in \mathcal{R}^+$, let $x(t)$ be a non-negative right-dense continuous function satisfying 
\begin{align*}
D^+_\Delta x(t) \leq -c x(t) + d \sup_{s \in [t-\t, t]_\mathbb{T}}x(s), \ t \in [0, \infty)_\mathbb{T}, 
\end{align*}
where 
$D^+_\Delta x(t)$ is the upper right Dini-delta-derivative of $x$ at $t$.
Then the inequality 
\begin{align*}
x(t) \leq \sup_{s \in [t-\t, t]_\mathbb{T}}x(s) e_{\ominus \lambda}(t,0),
\end{align*}
holds, where $\lambda>0$ is a solution of the inequality $\lambda + d \exp(\lambda \t)<c$.
\end{lemma}


Now, we are ready to give the first main result of this article in the following theorem.

\begin{theorem} \label{main-theorem-1}
  
Let \Cref{ass:fjgi-bounded} hold. If, for some $p \in \{1,2,\infty\}$, there exist a non-singular matrix $Z$ and a control gain matrix $K$ such that 
$\mathcal{M}_1^p - \mathcal{M}_2^p >0 $
 and $-\mathcal{M}_1^p \in \mathcal{R}^+$, 
where 
\begin{align*}
\mathcal{M}_1^p & = - \big ( M_p(-ZKZ^{-1},\mathbb{T}) + \|Z\|_p \|Z^{-1}\|_p  ((M_\G L_\U + M_\U L_\G) \\
& \quad + (M_\G L_\f + M_\f L_\G)\| \P\|_p + L_\G \|I\|_p) \big ),\\
\mathcal{M}_2^p & = \|Z\|_p \|Z^{-1}\|_p (M_\G L_\f + M_\f L_\G)(\|\Q\|_p + \t \|\R\|_p)
\end{align*} 
and $M_p(\cdot,\mathbb{T})$ denotes the unified matrix-measure as defined in \Cref{mm-ts},
then the drive system \eqref{eq:main-D} and the  response system \eqref{eq:main-R} are exponentially lag-synchronized. 
\end{theorem}
\begin{proof} 
For any non-singular matrix $Z$, we define 
\begin{align*}
V(\e(t)) = \| Z \e(t)\|_p.
\end{align*}
Now, for any arbitrary point $t \in \mathbb{T}$, from the definition of $\mu(t)$, we have either $\mu(t) =0$ or $\mu(t)>0$.  
Therefore, we split the proof into the following two steps:\\
\textbf{Step 1:} When $\mu(t)>0$, then for any $t \in \mathbb{T}$, we have
\begin{align}
   \dfrac{\| Z \e(\sigma(t)) \|_p  -\| Z \e(t) \|_p  }{\mu(t)} & = \dfrac{1}{\mu(t)} \bigg \{ \|Z \e(t) + \mu(t) Z \e^\Delta(t) \|_p  -\| Z \e(t) \|_p  \bigg \} \nonumber\\
&  = \dfrac{1}{\mu(t)} \bigg \{  \| Z \e(t) + \mu(t) Z(-K\e(t) - \td\G(\e(t))[\td\U(\e(t)) - \P \td \f(\e(t)) \nonumber \\ 
& \quad \quad - \Q \td \f(\e(t-\t_1)) - \R \int_{t-\t_2}^t \td \f(\e(s)) \Delta s - I] )  \|_p - \|Z\e(t)\|_p \bigg \} \nonumber \\
&  \leq \dfrac{1}{\mu(t)} \big \{ \|Z \e(t) + \mu(t)(-ZK)\e(t)\|_p -\| Z \e(t) \|_p  \big \} \nonumber\\
& \quad \quad + \|Z \td\G(\e(t))\td\U(\e(t))\|_p + \|Z \td\G(\e(t))\P\td \f(\e(t))\|_p  + \|Z \td\G(\e(t)) I\|_p \nonumber \\
& \quad \quad +  \|Z \td\G(\e(t)) \Q\td \f(\e(t-\t_1))\|_p
+ \| Z \td\G(\e(t)) \R \int_{t-\t_2}^t \td \f(\e(s))\Delta s   \|_p. \label{eq:eq-cal-1}
\end{align}
Now, from the definition of $\td \G, \td\U, \td \f$ and 
  \Cref{ass:fjgi-bounded}, we have 
\begin{align} \label{eq:eq-cal-2}
\|\td \G(\e(t)) \td \U(\e(t))\|_p & = \|\G(\y(t))\U(\y(t)) -\G(\x(t-\beta))\U(\x(t-\beta))\|_p \nonumber \\
& \leq \|\G(\y(t))\U(\y(t)) - \G(\y(t))\U(\x(t-\beta)) \|_p  \nonumber \\ 
&\quad + \|\G(\y(t))\U(\x(t-\beta)) -\G(\x(t-\beta))\U(\x(t-\beta))\|_p \nonumber \\ 
& \leq (M_\G L_\U + M_\U L_\G) \| \e(t)\|_p . 
\end{align}
Similarly,  one can obtain
\begin{align} \label{eq:eq-cal-3}
\| \td \G(\e(t)) \P \td \f(\e(t))\|_p & = \| \G(\y(t)) \P \f(\y(t)) -\G(\x(t-\beta)) \P \f(\x(t-\beta))\|_p \nonumber \\
& \leq (M_\G L_\f + M_\f L_\G) \|\P\|_p \| \e(t)\|_p,
\end{align}
\begin{align} \label{eq:eq-cal-4}
\| \td \G(\e(t)) \Q \td \f(\e(t-\t_1))\|_p \leq (M_\G L_\f + M_\f L_\G) \|\Q\|_p \sup_{s \in [t-\t_1 ,t]_\mathbb{T}} \| \e(s)\|_p,
\end{align}
\begin{align} \label{eq:eq-cal-5}
\bigg \| \td \G(\e(t)) \R \int_{t-\t_2}^t \td \f(\e(s)) \Delta s \bigg \|_p \leq \t (M_\G L_\f + M_\f L_\G) \|\R\|_p \sup_{s \in [t-\t_2 ,t]_\mathbb{T}} \| \e(s)\|_p
\end{align}
and 
\begin{align} \label{eq:eq-cal-6}
\| \td \G(\e(t)) I \|_p \leq L_\G \|I\|_p \| \e(t)\|_p.
\end{align}
Now, from the inequalities \eqref{eq:eq-cal-1}, \eqref{eq:eq-cal-2}, \eqref{eq:eq-cal-3}, \eqref{eq:eq-cal-4}, \eqref{eq:eq-cal-5} and \eqref{eq:eq-cal-6}, we get
\begin{align*}
  \dfrac{\|Z \e(\sigma(t)) \|_p  -\| Z \e(t) \|_p}{\mu(t)} & \leq   \dfrac{\| \Id + \mu(t) (-ZKZ^{-1})\|_p -1 }{\mu(t)} \|Z\e(t) \|_p + \|Z\|_p L_\G \|I\|_p  \|\e(t)\|_p
\\
& \quad \quad + \|Z\|_p (M_\G L_\U + M_\U L_\G) \| \e(t)\|_p   \\
& \quad \quad +  \|Z\|_p (M_\G L_\f + M_\f L_\G) \|\P\|_p \|\e(t)\|_p    \\
& \quad  \quad+ \|Z\|_p (M_\G L_\f + M_\f L_\G) \|\Q\|_p  \sup_{s \in [t-\t_1 ,t]_\mathbb{T}} \|Z \e(s)\|_p \\
& \quad \quad + \t \|Z\|_p (M_\G L_\f + M_\f L_\G) \|\R\|_p \sup_{s \in [t-\t_2 ,t]_\mathbb{T}} \|\e(s)\|_p \\
&  \leq  ( M_p(-ZK\P^{-1},\mathbb{T}) + \|Z\|_p L_\G \|I\|_p \|Z^{-1}\|_p  \| Z\e(t)\|_p \\
& \quad \quad + \|Z\|_p (M_\G L_\U + M_\U L_\G) \|Z^{-1}\|_p  \| Z \e(t)\|_p   \\
& \quad \quad +  \|Z\|_p (M_\G L_\f + M_\f L_\G) \|\P\|_p \| Z^{-1}\|_p  \|Z \e(t)\|_p    \\
& \quad \quad + \|Z\|_p (M_\G L_\f + M_\f L_\G) \|\Q\|_p \| Z^{-1}\|_p \sup_{s \in [t-\t_1 ,t]_\mathbb{T}} \|Z \e(s)\|_p \\
& \quad \quad + \t \|Z\|_p (M_\G L_\f + M_\f L_\G) \|\R\|_p \|Z^{-1}\|_p \sup_{s \in [t-\t_2 ,t]_\mathbb{T}} \| Z\e(s)\|_p \\
&  \leq -\mathcal{M}_1^p \|Z \e(t)\|_p  + \mathcal{M}_2^p \sup_{s \in [t-\t, t]_\mathbb{T}} \|Z \e(s)\|_p.
\end{align*}
Hence, using \Cref{def:delta}, we get 
\begin{align} \label{eq:eq-cal-7}
D^+_\Delta V(\e(t))  
  \leq   -\mathcal{M}_1^p V(\e(t)) + \mathcal{M}_2^p \sup_{s \in [t-\t, t]_\mathbb{T}} V(\e(s)).
\end{align}
\noindent \textbf{Step 2:} When $\mu(t)=0$, the derivative is the classical derivative, therefore, by using the formula $\x(t+h) = \x(t) + \x'(t)h + o(h)$ with $\lim _{h\to 0} \frac{\|o(h)\|_p}{h} = 0$, we can calculate
\begin{align*}
  \lim_{h \to 0^+} \dfrac{\|Z \e(t+h) \|_p -\| Z \e(t) \|_p }{h}  &  = \lim_{h \to 0^+}  \dfrac{1}{h} \bigg \{ \| Z \e(t) + h Z \e^\Delta(t) + o(h) \|_p -\| Z \e(t) \|_p  \bigg \} \\
& = \lim_{h \to 0^+}  \dfrac{1}{h} \bigg \{ \|Z \e(t) + h Z(-K\e(t) - \td\G(\e(t))[\td\U(\e(t)) - \P \td \f(\e(t)) \nonumber \\ 
& \quad \quad  - \Q \td \f(\e(t-\t_1)) - \R \int_{t-\t_2}^t \td \f(\e(s)) \Delta s - I] ) + o(h) \|_p - \|Z\e(t)\|_p \bigg \} \nonumber \\
&  \leq  ( M_p(-ZK\P^{-1},\mathbb{T}) + \|Z\|_p L_\G \|I\|_p \|Z^{-1}\|_p  \| Z\e(t)\|_p \\
& \quad \quad  + \|Z\|_p (M_\G L_\U + M_\U L_\G) \|Z^{-1}\|_p  \| Z \e(t)\|_p   \\
& \quad \quad +  \|Z\|_p (M_\G L_\f + M_\f L_\G) \|\P\|_p \| Z^{-1}\|_p  \|Z \e(t)\|_p    \\
& \quad \quad + \|Z\|_p (M_\G L_\f + M_\f L_\G) \|\Q\|_p \| Z^{-1}\|_p \sup_{s \in [t-\t_1 ,t]_\mathbb{T}} \|Z \e(s)\|_p \\
& \quad \quad + \t \|Z\|_p (M_\G L_\f + M_\f L_\G) \|\R\|_p \|Z^{-1}\|_p \sup_{s \in [t-\t_2 ,t]_\mathbb{T}} \| Z\e(s)\|_p \\
&  \leq -\mathcal{M}_1^p \|Z \e(t)\|_p  + \mathcal{M}_2^p \sup_{s \in [t-\t, t]_\mathbb{T}} \|Z \e(s)\|_p.
\end{align*}
Hence, using \Cref{def:delta} again, we get the same inequality as \eqref{eq:eq-cal-7}.

Thus, from the above two steps, for any $t \in \mathbb{T}$, we have
\begin{align*}
D^+ _\Delta V(\e(t))  
  \leq   -\mathcal{M}_1^p V(\e(t)) + \mathcal{M}_2^p \sup_{s \in [t-\t, t]_\mathbb{T}} V(\e(s)).
\end{align*}
Therefore, from  \Cref{lemma-1}, we get 
\begin{align*}
V(\e(t)) \leq  \sup_{s \in [t-\t,t]_\mathbb{T}} V(\e(s)) e_{\ominus \lambda} (t,0),
\end{align*}
where $\lambda$ is the solution of $\lambda +  \mathcal{M}_2^p \exp(\lambda \t) \leq \mathcal{M}_1^p$.
Further, it is clear that 
\begin{align*}
\| \e(t) \|_p & = \| Z^{-1} Z \e(t) \|_p \\
& \leq \|Z^{-1}\|_p \|V(\e(t))\|_p \\
& \leq \|Z^{-1}\|_p \sup_{s \in [t-\t,t]_\mathbb{T}} V(\e(s)) e_{\ominus \lambda }(t,0) \\
& \leq C e_{\ominus \lambda }(t,0),
\end{align*}
where $C= \| Z\|_p \|Z^{-1}\|_p \sup_{s \in [t-\t,t]_\mathbb{T}} \|\e(s)\|>0$. 
 Hence, from \Cref{main:Def:ELS}, the error 
   system   \eqref{eq:main-E} is exponentially stable, and hence, 
the drive system \eqref{eq:main-D} and the response system \eqref{eq:main-R} are exponentially lag-synchronized. 
\end{proof}
\begin{remark} \label{remark:p:id}
By choosing $Z=\Id$, the constants $\mathcal{M}_1^p$ and $\mathcal{M}_2^p$ of \Cref{main-theorem-1} become
\begin{align*}
\mathcal{M}_1^p & = - (M_p(-K,\mathbb{T}) + (M_\G L_\U + M_\U L_\G) + (M_\G L_\f + M_\f L_\G)\| \P\|_p + L_\G \|I\|_p), \\
\mathcal{M}_2^p & = (M_\G L_\f + M_\f L_\G)(\|\Q\|_p + \t \|\R\|_p).
\end{align*} 
\end{remark}
Next, we consider a particular case of the considered problem by setting   
$\G(\x(t)) = \Id$ and $\U(\x(t)) = \A \x(t)$, where $\A=\diag\{\a_1,\a_2,\ldots,\a_n \} \in \mathbb{R}^{n \times n}$ with $\a_i>0, i=1,2,\ldots,n$,
 then the drive system \eqref{eq:main-D} and the response system  \eqref{eq:main-R} become 
\begin{align} \label{eq:main-D-NN}
\begin{cases}
\x^\Delta(t) & =  - \A \x(t) + \P \f(\x(t)) + \Q \f(\x(t-\t_1)) + \R \int_{t-\t_2}^t \f(\x(s)) \Delta s + I, \  t \in [0,\infty)_\mathbb{T},\\
\x(s)  & = \phi(s), \ s \in [-\t , 0]_\mathbb{T}
\end{cases}
\end{align}
and 
\begin{align} \label{eq:main-R-NN}
\begin{cases}
\y^\Delta(t)  & = - \A \y(t) + \P \f(\y(t)) + \Q \f(\y(t-\t_1)) + \R \int_{t-\t_2}^t \f(\y(s)) \Delta s  + I + u(t), \ t \in [0,\infty)_\mathbb{T},\\
\y(s)  & = \psi(s), \ s \in [-\t , 0]_\mathbb{T}, 
\end{cases}
\end{align}
respectively. 
Also, the error system \eqref{eq:main-E} becomes
\begin{align} \label{eq:main-E-NN}
\e^\Delta(t)  = -(\A+K)\e(t)  + \P \hat \f(\e(t)) + \Q \hat \f(\e(t-\t_1)) + \R \int_{t-\t_2}^t \hat \f(\e(s)) \Delta s ,
\end{align}
where $\hat \f(\e(\cdot)) = \f(\y(\cdot)) -\f(\x(\cdot-\beta))$.

\begin{remark} One could have a remark similar to \Cref{remark:ts:cases} for the drive system \eqref{eq:main-D-NN} and the response system \eqref{eq:main-R-NN}.
\end{remark}

Now, we will give some sufficient conditions for the exponential lag synchronization for the systems \eqref{eq:main-D-NN}--\eqref{eq:main-R-NN} as follows.

\begin{theorem} \label{main-theorem-1-NN}
 
Let
$\f$ satisfy the Lipschitz and bounded conditions as stated in \Cref{ass:fjgi-bounded}. If, for some $p \in \{1,2,\infty\}$, there exist a non-singular matrix $Z$ and a control gain matrix $K$ such that 
$\mathcal{M}_3^p - \mathcal{M}_4^p >0 $
 and $-\mathcal{M}_3^p \in \mathcal{R}^+$, 
where 
\begin{align*}
\mathcal{M}_3^p & = - (M_p(-Z(\A+K)Z^{-1},\mathbb{T}) + \|Z\|_p \|Z^{-1}\|_p \|\P\|_p L_\f  ),\\
\mathcal{M}_4^p & = \|Z\|_p \|Z^{-1}\|_p L_\f(\|\Q\|_p + \t \|\R\|_p),
\end{align*}
then the drive system \eqref{eq:main-D-NN} and response system \eqref{eq:main-R-NN} are exponentially lag-synchronized. 
 
\end{theorem}
\begin{proof}
For any non-singular matrix $Z$, we define 
\begin{align*}
V(\e(t)) = \| Z \e(t)\|_p.
\end{align*}
  Similar to the proof of \Cref{main-theorem-1}, we
  consider the following two steps:\\
\textbf{Step 1:} When $\mu(t)>0$, then for any $t \in \mathbb{T}$, we have
\begin{align*}
   \dfrac{\| Z \e(\sigma(t)) \|_p -\| Z \e(t) \|_p  }{\mu(t)} 
 &  = \dfrac{1}{\mu(t)} \bigg \{ \|Z \e(t) + \mu(t) Z \e^\Delta(t) \|_p  -\| Z \e(t) \|_p  \bigg \} \nonumber\\
&  = \dfrac{1}{\mu(t)} \bigg \{  \| Z \e(t) + \mu(t) Z(-(\A+K)\e(t)  + \P \hat \f(\e(t)) + \Q \hat \f(\e(t-\t_1)) \\
& \quad \quad + \R \int_{t-\t_2}^t \hat \f(\e(s)) \Delta s )  \|_p - \|Z\e(t)\|_p \bigg \} \nonumber \\
&  \leq \dfrac{1}{\mu(t)} \big \{ \|Z \e(t) + \mu(t)(-Z(\A+K))\e(t)\|_p -\| Z \e(t) \|_p  \big \} + \|Z \P \hat \f(\e(t)) \|_p\nonumber\\
& \quad \quad  + \|Z \Q \hat \f(\e(t-\t_1))\|_p  + \|Z \R \int_{t-\t_2}^t \hat \f(\e(s)) \Delta s  \|_p \nonumber   \\ 
& \quad \leq -\mathcal{M}_3^p \|Z \e(t)\|_p  + \mathcal{M}_4^p \sup_{s \in [t-\t, t]_\mathbb{T}} \|Z \e(s)\|_p. 
\end{align*}
Hence, from \Cref{def:delta}, we get 
\begin{align} \label{eq:eq-cal-7-NN}
D^+_\Delta V(\e(t)) 
   \leq   -\mathcal{M}_3^p V(\e(t)) + \mathcal{M}_4^p \sup_{s \in [t-\t, t]_\mathbb{T}} V(\e(s)).
\end{align}
\noindent \textbf{Step 2:} When $\mu(t)=0$, then, using the same analysis as in Step 1, we get  
\begin{align*}
\lim_{h \to 0^+} \dfrac{\|Z \e(t+h) \|_p -\| Z \e(t) \|_p }{h} 
&  = \lim_{h \to 0^+}  \dfrac{1}{h} \bigg \{ \| Z \e(t) + h Z \e^\Delta(t) + o(h) \|_p -\| Z \e(t) \|_p  \bigg \} \\
&  \leq \lim_{h \to 0^+}  \dfrac{1}{h} \bigg \{ \|Z \e(t) + h Z(-(\A+K)\e(t)  + \P \hat \f(\e(t)) \\
& \quad \quad + \Q \hat \f(\e(t-\t_1))  + \R \int_{t-\t_2}^t \hat \f(\e(s)) \Delta s ) + o(h)  \|_p - \|Z\e(t)\|_p \bigg \} \nonumber \\
&  \leq -\mathcal{M}_3^p \|Z \e(t)\|_p  + \mathcal{M}_4^p \sup_{s \in [t-\t, t]_\mathbb{T}} \|Z \e(s)\|_p.
\end{align*}
Hence, using \Cref{def:delta} again, we get the same inequality as \eqref{eq:eq-cal-7-NN}.

Thus, from the above two steps, for any $t \in \mathbb{T}$, we have
\begin{align*}
D^+ _\Delta V(\e(t))  
  \leq   -\mathcal{M}_3^p V(\e(t)) + \mathcal{M}_4^p \sup_{s \in [t-\t, t]_\mathbb{T}} V(\e(s)).
\end{align*}
Therefore, from  \Cref{lemma-1}, we get 
$
V(\e(t)) \leq  \sup_{s \in [t-\t,t]_\mathbb{T}} V(\e(s)) e_{\ominus \lambda} (t,0),
$
where $\lambda$ is the solution of $\lambda +  \mathcal{M}_4^p \exp(\lambda \t) \leq \mathcal{M}_3^p$.
Further, it is clear that 
$\| \e(t) \|_p  = \| Z^{-1} Z \e(t) \|_p \leq  C e_{\ominus \lambda }(t,0),
$
where $C= \| Z\|_p \|Z^{-1}\|_p \sup_{s \in [t-\t,t]_\mathbb{T}} \|\e(s)\|>0$. 
 Hence, from \Cref{main:Def:ELS}, the error 
    system   \eqref{eq:main-E} is exponentially stable, and hence, 
the drive system \eqref{eq:main-D} and the response system \eqref{eq:main-R} are exponentially lag-synchronized. 
\end{proof}
\begin{remark} \label{remark:p:id:NN}
Similar to \Cref{remark:p:id}, by choosing $Z=\Id$, the constants $\mathcal{M}_3^p$ and $\mathcal{M}_4^p$ of \Cref{main-theorem-1-NN} become
\begin{align*}
\mathcal{M}_3^p  = - (M_p(-(\A+K),\mathbb{T}) +  \|\P\|_p L_\f)  , \ \mathcal{M}_4^p  = L_\f(\|\Q\|_p + \t \|\R\|_p).
\end{align*} 
\end{remark}
\begin{remark}
In the case when there is no distributed time-delay in the systems \eqref{eq:main-D}--\eqref{eq:main-R} (or \eqref{eq:main-D-NN}--\eqref{eq:main-R-NN}), i.e., when $\t_2 = 0$, then one
can establish all the above results by setting the corresponding terms to zero in the computation of
the constants $\mathcal{M}_1^p$ and $\mathcal{M}_2^p$ (or $\mathcal{M}_3^p$ and $\mathcal{M}_4^p$). 
\end{remark}
\begin{remark}
The results of  \Cref{main-theorem-1} and \Cref{main-theorem-1-NN} cover the problem in all generality, 
therefore, one can obtain the results for particular time domains, such as the continuous-time domain (when $\mathbb{T}=\mathbb{R}$) and discrete-time domain (when  $\mathbb{T}= \mathbb{Z}$), by replacing the matrix-measures evolves in the constants $\mathcal{M}_1^p,\mathcal{M}_2^p,\mathcal{M}_3^p$ and $\mathcal{M}_4^p$ from the known \Cref{mm-real}.    
\end{remark}

\begin{remark}
For the continuous-time domain, few authors reported the synchronization results for the C-GNNs with mixed delays \cite{work-lag-synchro-2,syn-lag, syn-finite, syn-adaptive}.
Particularly, in \cite{syn-lag}, the authors considered a class of
C-GNNs with mixed delays and studied the exponential lag synchronization via periodically intermittent control and mathematical induction technique.
In \cite{syn-finite}, the authors studied finite-time synchronization of C-GNNs with mixed delays by using the Lyapunov-Krasovskii functional approach.
Furthermore, there are only a few authors who studied the synchronization problem of the discrete-time C-GNNs \cite{CG-D-3,CG-D-7}. 
In particular, the authors in \cite{CG-D-7}, studied the exponential synchronization results for an array of coupled discrete-time
C-GNNs with time-dependent delay by applying the Lyapunov-Krasovskii functional approach while in \cite{CG-D-3}, the authors investigated the existence of a bounded unique solution, exponential stability, and synchronization by using some fixed point techniques and inequality techniques.
\end{remark}

\begin{remark}
All the results obtained on continuous-time \cite{work-lag-synchro-2,syn-lag, syn-finite, syn-adaptive} and discrete-time \cite{CG-D-3,CG-D-7} C-GNNs  are studied separately.
The continuous-time or discrete-time C-GNNs results cannot be directly applied and easily extended to the case of arbitrary time C-GNNs. 
And,  there is no manuscript on the continuous-time or discrete-time domain which discussed the exponential lag synchronization results for the C-GNNs with mixed delays
by using the matrix-measure and Halanay inequality, therefore, the results of this manuscript are completely new even for the continuous case ($\mathbb{T} = \mathbb{R}$) and discrete case ($\mathbb{T} = \mathbb{Z}$). 
\end{remark}





\section{ Illustrated Examples} \label{sec:examples}
 
In this section, we provide two examples to illustrate the obtained results for different time domains. Whereas the first example is tailored to best illustrate the potentials of our theoretical results with respect
to arbitrary time domains, the second example is borrowed from \cite{add-3} to show the general applicability of our methods.
 
\begin{example} \label{ex-1}
Consider the drive system \eqref{eq:main-D} and response system \eqref{eq:main-R} with the following coefficients
\begin{align*}
& \G(\x(t)) = \begin{bmatrix}
0.4 + 0.2 \cos(\x_1(t)) & 0.0 \\
0.0 & 0.4-0.2 \sin(\x_2(t))
\end{bmatrix}, \ 
\U(\x(t)) = \begin{bmatrix}
0.3 + 0.2 \sin(\x_1(t)) \\
0.3 - 0.2 \cos(\x_2(t))
\end{bmatrix},\\
&
\P = \begin{bmatrix}
0.8 & 0.0 \\
-0.2 & -0.7
\end{bmatrix}, \ 
\Q = \begin{bmatrix}
-0.4 & 0.1 \\
-0.2 & 0.5
\end{bmatrix},
\ 
\R = \begin{bmatrix}
-0.5 & 0.6 \\ 
-0.6 & 0.5
\end{bmatrix}, \ 
I = \begin{bmatrix}
0.4 \\ 
0.3
\end{bmatrix}, \\
&
\f(\x(t))=
\begin{bmatrix}
0.8 \tanh (\x_1(t)) \\ 0.8 \tanh (\x_2(t)) 
\end{bmatrix},
\ 
\phi(s) = 
\begin{bmatrix}
0.5 \\ 1 
\end{bmatrix},\
\psi(s) =
\begin{bmatrix}
-1 \\ -0.5 
\end{bmatrix} \ \text{for }
s \in  [-\t,0]_\mathbb{T}, \ Z=\Id.
\end{align*}
\end{example}

One can confirm that for \Cref{ex-1}, $\G, \U$, and $\f$ satisfy
\Cref{ass:fjgi-bounded} with $L_\G = L_\U = 0.2, L_\f =M_\f = 0.8, M_\G =0.6, M_\U = 0.5$.
Now, we consider the following three different time domains as follows.\\
\noindent \textbf{Case 1.} $\mathbb{T} = \mathbb{R}$.
 Let $\t_1 = 0.5, \t_2=0.8$ and $\beta = 0.4$. Here, $\t=0.8$ and the graininess function $\mu(t)=0$ for all  $ t \in \mathbb{R}$.
The state trajectories and the error trajectories of the systems \eqref{eq:main-D}--\eqref{eq:main-R} without feedback control are shown in Fig. \ref{fig:1} and Fig. \ref{fig:2}, respectively. Clearly, from Fig. \ref{fig:1} and Fig. \ref{fig:2}, the drive system \eqref{eq:main-D} and the response system \eqref{eq:main-R} are not synchronized.  
\newlength\figureheight
\newlength\figurewidth
\setlength\figureheight{4cm}
\setlength\figurewidth{8cm}
\graphicspath{{plot-real/}}
\begin{figure}
\centering
\begin{minipage}{.5\textwidth}
  \centering
  \includegraphics[width=1\linewidth]{2x2Uncoupled-System}
  \caption{Uncoupled synchronization curves \\ when $\mathbb{T} = \mathbb{R}$ }
  \label{fig:1}
\end{minipage}%
\begin{minipage}{.5\textwidth}
  \centering
  \includegraphics[width=1.\linewidth]{2x2Uncoupled-System-lse}
  \caption{Uncoupled synchronization error \\ curves  when $\mathbb{T} = \mathbb{R}$}
  \label{fig:2}
\end{minipage}
\end{figure}
However, for the control gain matrix
$$ 
K = \begin{bmatrix}
2.2 & 0.0\\
0.0 & 2.2
\end{bmatrix},$$ 
we can calculate 
$$
\mathcal{M}_2^1 = 0.9472, \ \mathcal{M}_2^2 = 0.9542, \ 
\mathcal{M}_2^\infty = 1.0112 
$$
and 
\begin{align*}
  \Lambda_1(-K) &= -2.2000, \ \Lambda_2(-K) = -4.4000,\    \Lambda_\infty(-K) = -2.2000.
\end{align*}  
 Hence, 
 $$
 \mathcal{M}_1^1  = 0.7800 ,\quad \mathcal{M}_1^2 = 3.2242, \quad \text{and } \mathcal{M}_1^\infty = 1.0840.
 $$
Therefore, we can see that 
$\mathcal{M}_1^1 - \mathcal{M}_2^1 = -0.1672 <0,$  
$\mathcal{M}_1^2 - \mathcal{M}_2^2 = 2.2700 >0,$ 
and  
$\mathcal{M}_1^\infty - \mathcal{M}_2^\infty =  0.0728 >0$. 
Also, $-\mathcal{M}_1^2, -\mathcal{M}_1^\infty \in \mathcal{R}^+$.
Hence, for $p=2, \infty$, all the conditions of \Cref{main-theorem-1} hold,
and 
thus, the systems \eqref{eq:main-D}--\eqref{eq:main-R} with feedback control \eqref{eq:control} are exponentially lag-synchronized with the maximum
rate of convergence for $p = 2,\infty$ are $1.0366$ and $0.0394$, respectively.
The synchronized curves and synchronized errors curves with feedback control are shown in  Fig. \ref{fig:3} and Fig. \ref{fig:4}, respectively.
\begin{figure}
\centering
\begin{minipage}{.5\textwidth}
  \centering
  \includegraphics[width=1\linewidth]{2x2Coupled-System}
  \caption{Coupled synchronization curves \\ when $\mathbb{T} = \mathbb{R}$}
  \label{fig:3}
\end{minipage}%
\begin{minipage}{.5\textwidth}
  \centering
  \includegraphics[width=1.\linewidth]{2x2Coupled-System-lse}
  \caption{Coupled synchronization error curves \\ when $\mathbb{T} = \mathbb{R}$}
  \label{fig:4}
\end{minipage}
\end{figure}

\noindent \textbf{Case 2.} $\mathbb{T} = 0.5 \mathbb{Z}$.
 Let $\t_1=\t_2=\beta = 0.5$. 
Here, $\t=0.5$ and the graininess function $\mu(t)=0.5$ for all  $ t \in \mathbb{R}$.
The state trajectories and the error trajectories of the systems \eqref{eq:main-D}--\eqref{eq:main-R} without feedback control are shown in Fig. \ref{fig:5} and Fig. \ref{fig:6}, respectively which are clearly not synchronized.  
\graphicspath{{plot-discrete/}}
\begin{figure}
\centering
\begin{minipage}{.5\textwidth}
  \centering
  \includegraphics[width=1\linewidth]{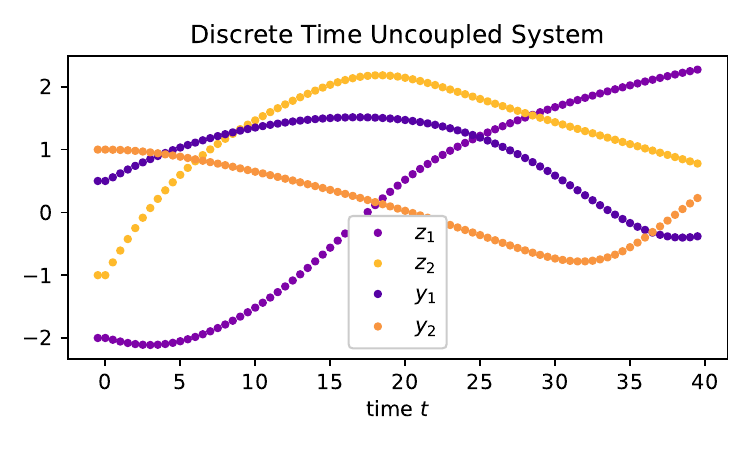}
  \caption{Uncoupled synchronization curves \\ when $\mathbb{T} = \frac{1}{2}\mathbb{Z}$}
  \label{fig:5}
\end{minipage}%
\begin{minipage}{.5\textwidth}
  \centering
  \includegraphics[width=1.\linewidth]{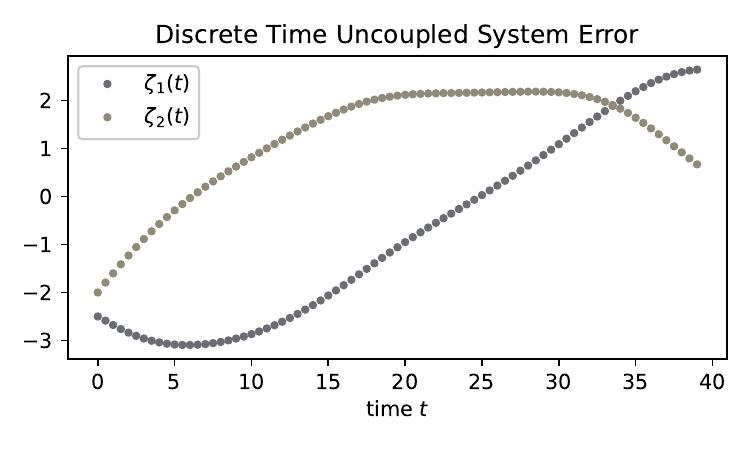}
  \caption{Uncoupled synchronization error   \\ curves when $\mathbb{T} = \frac{1}{2}\mathbb{Z}$}
  \label{fig:6}
\end{minipage}
\end{figure}
However, for the control gain matrix
$$ 
K = \begin{bmatrix}
2.0 & 0.0\\
0.0 & 2.0
\end{bmatrix},$$ 
we can calculate 
$$
\mathcal{M}_2^1 = 0.7360, \ \mathcal{M}_2^2 = 0.7430, \ 
\mathcal{M}_2^\infty = 0.8000 
$$
and 
\begin{align*}
  \Lambda_1(-K) &= -2.000, \ \Lambda_2(-K) = -2.000,\    \Lambda_\infty(-K) = -2.000.
\end{align*}  
 Hence, 
 $$
 \mathcal{M}_1^1  = 0.5800 ,\quad \mathcal{M}_1^2 = 0.8242, \quad \text{and } \mathcal{M}_1^\infty = 0.8840.
 $$
Therefore, we can see that 
$\mathcal{M}_1^1 - \mathcal{M}_2^1 = -0.1560 <0,$  
$\mathcal{M}_1^2 - \mathcal{M}_2^2 = 0.0812 >0,$ 
and  
$\mathcal{M}_1^\infty - \mathcal{M}_2^\infty =  0.0840 >0$. 
Also, $-\mathcal{M}_1^2, -\mathcal{M}_1^\infty \in \mathcal{R}^+$.
Hence, for $p=2, \infty$, all the conditions of \Cref{main-theorem-1} hold, 
and 
thus, the systems \eqref{eq:main-D}--\eqref{eq:main-R} with feedback control \eqref{eq:control} are exponentially lag-synchronized with the maximum
rate of convergence for $p = 2,\infty$ are $0.0583$ and $0.0590$, respectively.
The synchronized curves and synchronized errors curves with feedback control are shown in  Fig. \ref{fig:7} and Fig. \ref{fig:8}, respectively.
\begin{figure}
\centering
\begin{minipage}{.5\textwidth}
  \centering
  \includegraphics[width=1\linewidth]{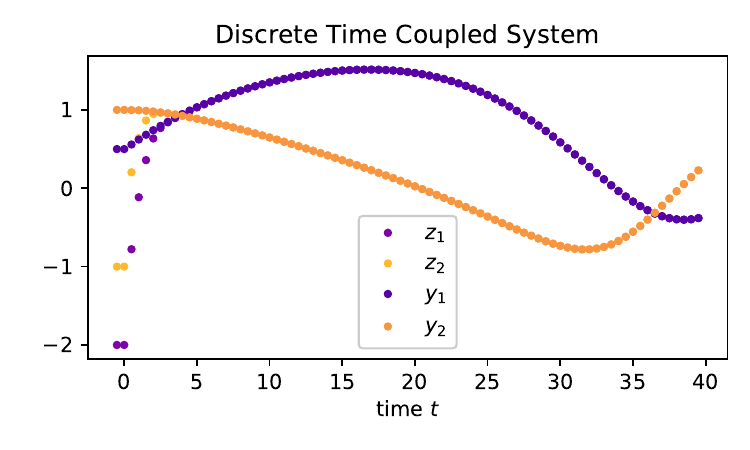}
  \caption{Coupled synchronization curves  \\ when $\mathbb{T} = \frac{1}{2}\mathbb{Z}$}
  \label{fig:7}
\end{minipage}%
\begin{minipage}{.5\textwidth}
  \centering
  \includegraphics[width=1.\linewidth]{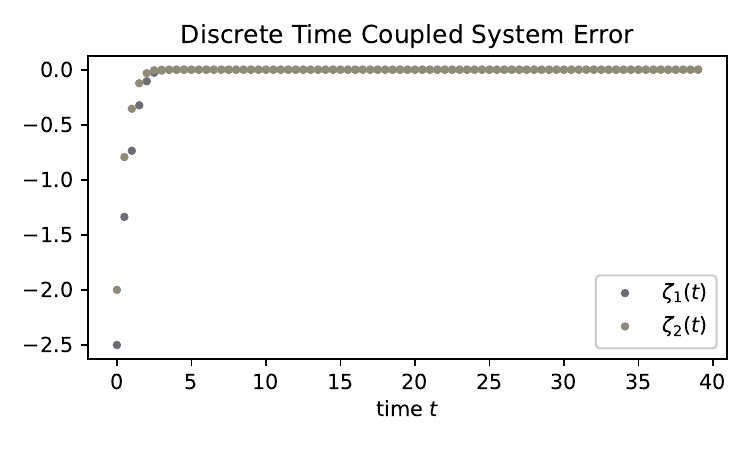}
  \caption{Coupled synchronization error curves  \\ when $\mathbb{T} = \frac{1}{2}\mathbb{Z}$}
  \label{fig:8}
\end{minipage}
\end{figure}

\noindent \textbf{Case 3.} $\mathbb{T} = \mathcal{P} = [-1,0] \cup_{i=0}^\infty [i, i+0.7]$.
 Let $\t_1=\t_2=\beta = 1$. 
 Here, $\t=1$ and the graininess function $\mu(t)$ is given by 
\begin{align*}
\mu(t) = 
\begin{cases}
0, \ t \in [-1,0] \cup_{i=0}^\infty [i, i+0.7), \\
0.3, \ t= \cup_{i=0}^\infty \{i+0.7 \}.
\end{cases}
\end{align*}  
The state trajectories and the error trajectories of the systems \eqref{eq:main-D}--\eqref{eq:main-R} without feedback control are shown in Fig. \ref{fig:9} and Fig. \ref{fig:10}, respectively which are clearly not synchronized.  
\graphicspath{{plot-ts-1/}}
\begin{figure}
\centering
\begin{minipage}{.5\textwidth}
  \centering
  \includegraphics[width=1\linewidth]{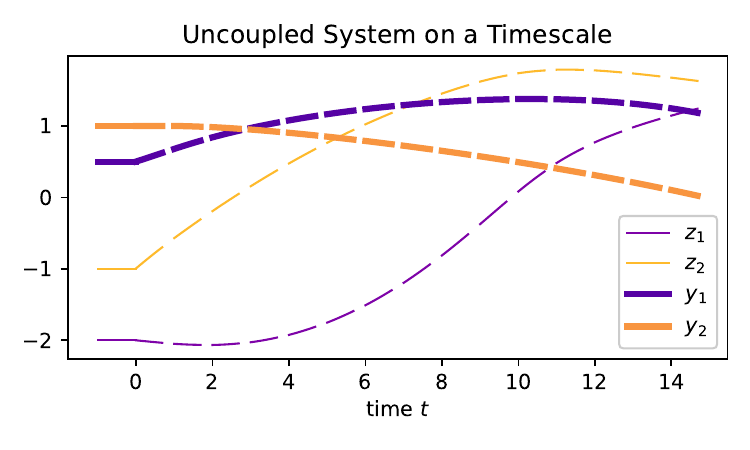}
  \caption{Uncoupled synchronization curves \\ when $\mathbb{T} = \mathcal{P}$}
  \label{fig:9}
\end{minipage}%
\begin{minipage}{.5\textwidth}
  \centering
  \includegraphics[width=1.\linewidth]{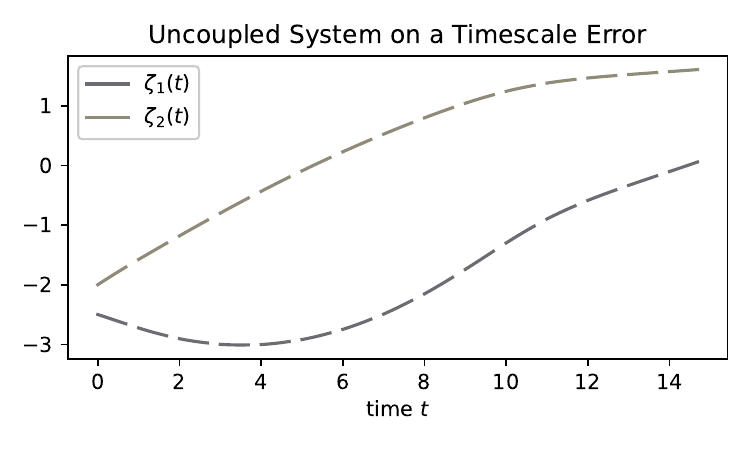}
  \caption{Uncoupled synchronization error curves  when $\mathbb{T} = \mathcal{P}$}
  \label{fig:10}
\end{minipage}
\end{figure}
However, for the control gain matrix
$$ 
K = \begin{bmatrix}
2.4 & 0.0\\
0.0 & 2.4
\end{bmatrix},$$ 
we can calculate 
$$
\mathcal{M}_2^1 = 1.0880, \ \mathcal{M}_2^2 = 1.0950, \ 
\mathcal{M}_2^\infty = 1.1520 
$$
and 
\begin{align*}
  \Lambda_1(-K) &= -2.4000, \ \Lambda_2(-K) = -2.4000,\    \Lambda_\infty(-K) = -2.4000.
\end{align*}  
 Hence, 
 $$
 \mathcal{M}_1^1  = 0.9800 ,\quad \mathcal{M}_1^2 = 1.2242, \quad \text{and } \mathcal{M}_1^\infty = 1.2840.
 $$
Therefore, we can see that 
$\mathcal{M}_1^1 - \mathcal{M}_2^1 = -0.1080 <0,$  
$\mathcal{M}_1^2 - \mathcal{M}_2^2 = 0.1292 >0,$ 
and  
$\mathcal{M}_1^\infty - \mathcal{M}_2^\infty =  0.1320 >0$.
Also, $-\mathcal{M}_1^2, -\mathcal{M}_1^\infty \in \mathcal{R}^+$. 
Hence, for $p=2, \infty$, all the conditions of \Cref{main-theorem-1} hold
and 
thus, the systems \eqref{eq:main-D}--\eqref{eq:main-R} with feedback control \eqref{eq:control} are exponentially lag-synchronized with the maximum
rate of convergence for $p = 2,\infty$ are $0.0602$ and $0.0599$, respectively.
The synchronized curves and synchronized errors curves with feedback control are shown in  Fig. \ref{fig:11} and Fig. \ref{fig:12}, respectively.
\begin{figure}
\centering
\begin{minipage}{.5\textwidth}
  \centering
  \includegraphics[width=1\linewidth]{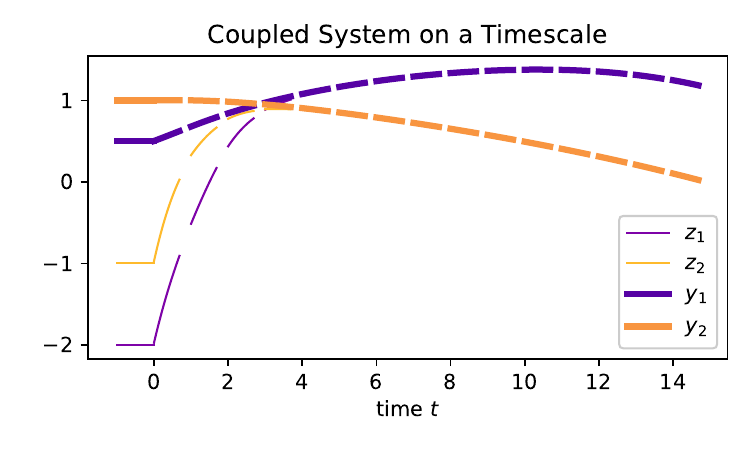}
  \caption{Coupled synchronization curves  \\ when $\mathbb{T} = \frac{1}{2}\mathbb{Z}$}
  \label{fig:11}
\end{minipage}%
\begin{minipage}{.5\textwidth}
  \centering
  \includegraphics[width=1.\linewidth]{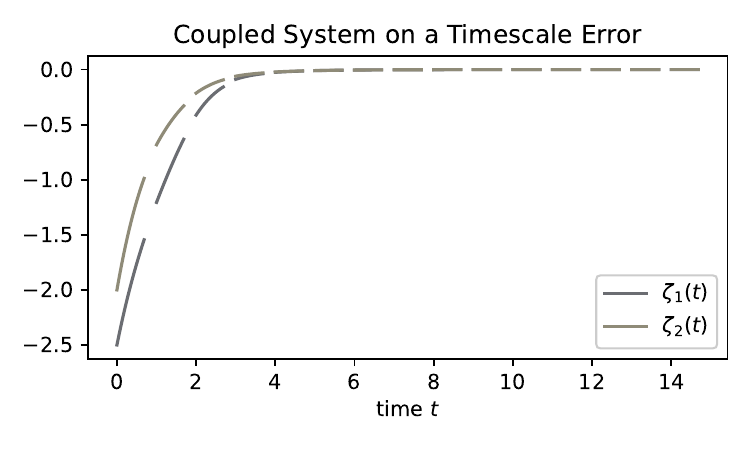}
  \caption{Coupled synchronization error curves  when $\mathbb{T} = \frac{1}{2}\mathbb{Z}$}
  \label{fig:12}
\end{minipage}
\end{figure}

Next, we provide another example to illustrate our main \Cref{main-theorem-1-NN}. 
\begin{example} \label{ex-3}
Consider the continuous-time case of the drive and response systems \eqref{eq:main-D-NN}--\eqref{eq:main-R-NN} with the following coefficients as in \cite[Ex. 2]{add-3}
\begin{align*}
&\A = \begin{bmatrix}
1 & 0 \\
0 & 1
\end{bmatrix}, \ \P = \begin{bmatrix}
2.0 & -0.1 \\
-5.0 & 2.8
\end{bmatrix}, \ 
\Q = \begin{bmatrix}
-1.6 & -0.1 \\ 
-0.3 & -2.5
\end{bmatrix},
\ 
\R = \begin{bmatrix}
0.5 & 0.6 \\
0.7 & 0.2
\end{bmatrix}, \ 
I = \begin{bmatrix}
0.0 \\ 
0.0
\end{bmatrix}, \\
&
\f(\x(t))=
\begin{bmatrix}
\tanh(\x_1(t))  \\
\tanh(\x_2(t)))
\end{bmatrix}, \ \t_1 = 0.6, \t_2=0.2, \beta= 0,
\\ 
& \phi(s) = 
\begin{bmatrix}
0.2 \\ 0.4 
\end{bmatrix},\
\psi(s) =
\begin{bmatrix}
-0.4 \\ -0.6 
\end{bmatrix} \ \text{for }
s \in  [-1,0]_\mathbb{T}, \ Z=\Id.
\end{align*}
\end{example}
One can confirm that for \Cref{ex-3}, $\f$ satisfies the Lipschitz conditions with $ L_\f =1$.
The state trajectories and the error trajectories of the systems \eqref{eq:main-D-NN}--\eqref{eq:main-R-NN} without feedback control are shown in Fig. \ref{fig:19} and Fig. \ref{fig:20}, respectively. Clearly, from Fig. \ref{fig:19} and Fig. \ref{fig:20}, the drive system \eqref{eq:main-D-NN} and the response system \eqref{eq:main-R-NN} are not synchronized.  
\graphicspath{{plot-ex-3/}}
\begin{figure}
\centering
\begin{minipage}{.5\textwidth}
  \centering
  \includegraphics[width=1\linewidth]{2x2Uncoupled-System}
  \caption{Uncoupled synchronization curves}
  \label{fig:19}
\end{minipage}%
\begin{minipage}{.5\textwidth}
  \centering
  \includegraphics[width=1.\linewidth]{2x2Uncoupled-System-lse}
  \caption{Uncoupled synchronization error \\curves }
  \label{fig:20}
\end{minipage}
\end{figure}
However, for the control gain matrix
$$ 
K = \begin{bmatrix}
3.8 & 0.0\\
0.0 & 3.8
\end{bmatrix},$$ 
we can calculate 
$$
\mathcal{M}_4^1 = 3.320, \ \mathcal{M}_4^2 = 3.1331, \ 
\mathcal{M}_4^\infty = 3.460, 
$$
and 
\begin{align*}
  \Lambda_1(-(\A+K)) &= -4.80, \ \Lambda_2(-(\A+K)) = -9.60,\    \Lambda_\infty(-(\A+K)) = -4.80.
\end{align*}  
 Hence, 
 $$
 \mathcal{M}_3^1  = -2.200,\quad \mathcal{M}_3^2 = 3.5892, \quad \text{and } \mathcal{M}_3^\infty = -3.0.
 $$
Therefore, we see that 
$\mathcal{M}_3^1 - \mathcal{M}_4^1 = -5.520 <0,$  
$\mathcal{M}_3^2 - \mathcal{M}_4^2 = 0.4561 >0,$ 
and  
$\mathcal{M}_3^\infty - \mathcal{M}_4^\infty =  -6.460 <0$. 
Also, $-\mathcal{M}_3^2 \in \mathcal{R}^+$.
Hence, for $p=2$, all the conditions of \Cref{main-theorem-1-NN} hold,
and 
thus, the systems \eqref{eq:main-D-NN}--\eqref{eq:main-R-NN} with feedback control \eqref{eq:control} are exponentially synchronized with the maximum
rate of convergence $0.1533$.
The synchronized curves and synchronized errors curves with feedback control are shown in  Fig. \ref{fig:21} and Fig. \ref{fig:22}, respectively.
\begin{figure}
\centering
\begin{minipage}{.5\textwidth}
  \centering
  \includegraphics[width=1\linewidth]{2x2Coupled-System}
  \caption{Coupled synchronization curves }
  \label{fig:21}
\end{minipage}%
\begin{minipage}{.5\textwidth}
  \centering
  \includegraphics[width=1.\linewidth]{2x2Coupled-System-lse}
  \caption{Coupled synchronization error curves}
  \label{fig:22}
\end{minipage}
\end{figure}

Comparing the results quantitatively, we note that our approach provides a faster error convergence rate $0.1533$, compared to the convergence rate of $0.01$ reported in \cite{add-3}.

\begin{remark}
Previous works, such as \cite{work-lag-synchro-2,syn-lag,add-3, syn-finite, syn-adaptive,CG-D-3,CG-D-7,add-1,add-2}, have considered similar types of examples 
on either continuous or discrete-time domains. 
To the best of our knowledge, there is currently no other example in the literature that has 
addressed lag synchronization of CGNNs on hybrid-type time domains (as presented in case 3 of \Cref{ex-1}).

\end{remark}

\section*{Conclusion}
We have successfully established the exponential lag synchronization results for a new class of C-GNNs with discrete and distributed time delays on arbitrary time domains by using the theory of time scales and feedback control law.  We have also studied some special cases of the considered problem. 
We mainly used a unified matrix-measure theory and Halanay inequality to establish these results. The obtained results are verified by
providing some simulated examples for different time domains including the continuous-time domain (case 1 of \Cref{ex-1}, \Cref{ex-3}),
discrete-time domain (case 2 of \Cref{ex-1}), and non-overlapping time domain (case 3 of \Cref{ex-1}).
Possible future research could concern an extension of the results to non-smooth though still bounded functions. 
Another potential future direction is to further investigate the stability and synchronization results for C-GNNs with delays and impulsive conditions on arbitrary time domains. This could include studying the effects of different types of delays, such as time-varying delays or distributed delays, on the synchronization of C-GNNs.
Additionally, it could be interesting to investigate the robustness and reliability of the synchronization results for C-GNNs with delays and stochastic effects on time scales. This could include studying the effects of random disturbances or noise on the synchronization of C-GNNs, and how the proposed approach can be modified to handle these types of effects.



\end{document}